\documentclass[twoside,10pt,a4paper]{amsart}

\usepackage[T1]{fontenc} 
\usepackage[ansinew]{inputenc}

\usepackage{amsthm}

\usepackage{graphics}
\usepackage{multirow}

\usepackage[all]{xy}

\DeclareMathOperator{\im}{im}
\DeclareMathOperator{\Id}{id}
\DeclareMathOperator{\id}{id}
\DeclareMathOperator{\Hom}{Hom}
\DeclareMathOperator{\Ext}{Ext}

\newcommand{\smatrix}[1]{\left(\!\begin{smallmatrix} #1\end{smallmatrix}\!\right)}

\newcommand{\Tate}{\widehat{\Ext}}
\newcommand{\class}{\mathcal{C}}

\newtheorem{lemma}[equation]{Proposition}
\newtheorem{definition}[equation]{Definition}
\newtheorem{satz}[equation]{Theorem}
\newtheorem{korollar}[equation]{Corollary}

\theoremstyle{remark}
\newtheorem{bemerkung}[equation]{Remark}
\newtheorem{remark}[equation]{Remark}

\DeclareMathOperator{\tr}{tr}
\DeclareMathOperator{\coker}{coker}

\newcommand{\hochschild}{H\! H}

\newcommand{\B}{\mathcal{B}}
\newcommand{\A}{\mathcal{A}}

\newcommand{\rem}[1]{}

\numberwithin{equation}{section}
\begin{document}
\SelectTips{cm}{} 

\title[Secondary multiplication in Tate Cohomology of generalized quaternion groups]{Secondary multiplication in Tate Cohomology\\ of generalized quaternion groups}
\author{Martin Langer}

\begin{abstract}
Let $k$ be a field and let $G$ be a finite group. By a theorem of D.~Benson, H.~Krause and S.~Schwede, there is a canonical element in the Hochschild cohomology of the Tate cohomology $\gamma_G\in \hochschild^{3,-1} \hat{H}^*(G)$ with the following property: Given any graded  $\hat{H}^*(G)$-module $X$, the image of $\gamma_G$ in $\Ext^{3,-1}_{\hat{H}^*(G)} (X,X)$ is zero if and only if $X$ is isomorphic to a direct summand of $\smash{\hat{H}^*(G,M)}$ for some $kG$-module $M$. In particular, if $\gamma_G=0$ then every module is a direct summand of a realizable $\hat{H}^*(G)$-module.

We prove that the converse of that last statement is not true by studying in detail the case of generalized quaternion groups. Suppose that $k$ is a field of characteristic $2$ and $G$ is generalized quaternion of order $2^n$ with $n\geq 3$. We show that $\gamma_G$ is non-trivial for all $n$, but there is an $\hat{H}^*(G)$-module detecting this non-triviality if and only if $n=3$.
\end{abstract}
\maketitle

\section{Introduction}
Let $k$ be a field, $G$ a finite group, and let $\hat{H}^*(G)$ denote the graded Tate cohomology algebra of $G$ over $k$. The starting point of this paper is the following theorem of D.~Benson, H.~Krause and S.~Schwede:
\begin{satz}\label{bkstheorem} \cite{bks} 
There exists a canonical element in Hochschild cohomology of $\hat{H}^*(G)$
\[ \gamma_G\in \hochschild^{3,-1} \hat{H}^*(G), \]
such that for any graded $\hat{H}^*(G)$-module $X$, the following are equivalent:
\begin{itemize}
\item[(i)] The image of $\gamma_G$ in $\Ext^{3,-1}_{\hat{H}^*(G)} (X,X)$ is zero.
\item[(ii)] There exists a $kG$-module $M$ such that $X$ is a direct summand of the graded $\hat{H}^*(G)$-module $\hat{H}^*(G,M)$.
\end{itemize}
\end{satz}
Let us call an $\hat{H}^*(G)$-module \emph{realizable} if it is isomorphic to a module of the form $\hat{H}^*(G,M)$ for some $kG$-module $M$. As an immediate consequence we get the following:
\begin{korollar}
If $\gamma_G=0$, then every $\hat{H}^*(G)$-module is a direct summand of a realizable module.
\end{korollar}
At this point it is natural to ask for the converse of that statement. That is, given the fact that $\gamma_G\neq 0$, is there some $\hat{H}^*(G)$-module detecting the non-triviality of $\gamma_G$? Theorem~\ref{bkstheorem} works more generally in the situation of differential graded algebras, and in that setup the converse of the corresponding corollary is known to the false: Benson, Krause and Schwede provide an example of a dg algebra $A$ such that the canonical class $\gamma_A\in\hochschild^{3,-1}(H^*A)$ is non-trivial, but every $H^*A$-module is realizable (see \cite{bks}, Proposition~5.16). Nevertheless, the author believes that the question whether there is such an example coming from Tate cohomology of groups is still open.

In this paper we will compute $\gamma_G$ explicitly for the generalized quaternion groups $G$. 
In what follows, let $t\geq 2$ be a power of $2$, and let $G=Q_{4t}$ be the group of generalized quaternions
\[ Q_{4t} = \left< g,h \,\mid\, g^t = h^2,\, ghg=h \right>. \]
Let $k$ be a field of characteristic $2$, and denote by $L=kG$ the group algebra of $G$ over $k$.
Then the Tate cohomology ring $\hat{H}^*(G)$ is well-known; it is given by
\[
\hat{H}^*(Q_{4t}) =\Tate_L^*(k,k) \cong \begin{cases}
   k[x,y,s^{\pm 1}]/(x^2+y^2=xy,y^3=0) & \text{if $t=2$, } \\
   k[x,y,s^{\pm 1}]/(x^2=xy,y^3=0) & \text{if $t\geq 4$, }
                  \end{cases}
\]
with degrees $|x|=|y|=1, |s|=4$ (see e.g.~\cite{CartEile},~Chapter~XII~\S~11, and \cite{AdemMilgram},~IV Lemma 2.10).
Our main goal is to prove the following theorem.
\begin{satz}\label{haupttheorem}
The element $\gamma_{Q_8}\in \hochschild^{3,-1} \hat{H}^*(Q_8)$ is non-trivial, and the cokernel of the map
\[ \hat{H}^*(Q_8)[-1] \oplus \hat{H}^*(Q_8)[-1] \xrightarrow{\smatrix{y & x+y \\ x & y}} \hat{H}^*(Q_8)\oplus \hat{H}^*(Q_8) \]
is a graded $\hat{H}^*(Q_8)$-module which is not a direct summand of a realizable one.
For $t\geq 4$ the element $\gamma_{Q_{4t}}\in \hochschild^{3,-1} \hat{H}^*(Q_{4t})$ is non-trivial, but
every graded $\hat{H}^*(Q_{4t})$-module is a direct summand of a realizable one.
\end{satz}
The plan is as follows: In the first section we will briefly recall the definitions needed in Theorem~\ref{bkstheorem}; most of this part is taken from \cite{bks}, and the reader interested in details should consult that source. In the second section we turn to the computation of a Hochschild cocycle $m$ representing the canonical class $\gamma_G$. In the third section we prove the statements about realizability of modules. Theorem~\ref{haupttheorem} will then follow from Theorems~\ref{mquat8}, \ref{mquat16}, \ref{mquat8module}, and Propositions~\ref{gammaprimecupmzero} and~\ref{gammadblprimecupmzero}.

\subsection*{Acknowledgments} Part of this paper is part of a Diploma thesis written at the Mathematical Institute, University of Bonn. I would like to thank my advisor Stefan Schwede for suggesting the subject and all the helpful comments on the project. I would also like to thank the referee of an earlier version of this paper for some helpful remarks.

\subsection{Notations and conventions}
All occurring modules will be right modules. We shall often work over a fixed ground field $k$; then $\otimes$ means tensor product over $k$. Whenever convenient, we write $(a_1,a_2,\dots,a_n)$ instead of $a_1\otimes a_2\otimes\dots\otimes a_n$. If $G$ is a group, then $k$ is often considered as a trivial $kG$-module. 

Let $R$ be a ring with unit, and let $M$ be a $\mathbb{Z}$-graded $R$-module. The degree of every (homogeneous) element $m\in M$ will be denoted by $|m|$. For every integer $n$ the module $M[n]$ is defined by $M[n]^j=M^{n+j}$ for all $j$. Given two such modules $M$ and $L$, a morphism $f:L\longrightarrow M$ is a family $f^j:L^j\longrightarrow M^j$ of $R$-module homomorphisms. The group of all these morphisms is denoted by $\Hom_R(L,M)$. Furthermore, we have $\Hom_R^m (L,M) = \Hom_R(L,M[m])$, the morphisms of degree $m$. The graded module $L\otimes M$ is given by $(L\otimes M)^m =\bigoplus_{i+j=m} L^i\otimes M^j$. If $M$ is a differential graded $R$-module with differential $d$, then the differential of $M[n]$ is given by $(-1)^n d$.

\subsection{Tate Cohomology}\label{tatekohomologie}
Let us recall briefly the definition and basic properties of Tate cohomology. Let $k$ be a field, and let $G$ be a finite group. Then $L=kG$ is a self-injective algebra (i.e.~the classes of projective and injective right-$L$-modules coincide). For any $L$-module $N$ we get a complete projective resolution $P_*$ of $N$ by splicing together a projective and an injective resolution of $N$:
\[
\xymatrix@!@=2pt{
\dots & &
P_{-2} \ar[ll] &&
P_{-1} \ar[ll] &&
P_{0} \ar[ll]\ar@{->>}[dl]  &&
P_{1} \ar[ll]  &&
P_{2} \ar[ll]  &&
\dots  \ar[ll]  \\
&&& & & N\ar@{_{(}->}[ul]
} \]
Given another $L$-module $M$, we can apply the functor $\Hom_L(-,M)$ to $P_*$; then Tate cohomology is defined to be the cohomology groups of the resulting complex:
\begin{equation}\label{tatedef} \Tate_L^n(N,M)=H^n(\Hom_L(P_*,M)) \quad\text{for all $n\in\mathbb{Z}$.}  \end{equation}
For arbitrary $L$-modules $X,Y$ and $Z$, we have a cup product
\begin{equation} \label{cupprodukt} \Tate_L^m(Y,Z)\otimes \Tate_L^n(X,Y) \longrightarrow \Tate_L^{m+n}(X,Z), \end{equation}
see e.g.~\cite{Carlson}, \S~6. Therefore, $\hat{H}^*(G)=\hat{H}^*(G,k)=\Tate_{kG}^*(k,k)$ is a graded algebra, and $\hat{H}^*(G,M)=\Tate_{kG}^*(k,M)$ is a graded $\hat{H}^*(G)$-module for every $kG$-module $M$. We call a graded $\hat{H}^*(G)$-module $X$ \emph{realizable} if it is isomorphic to $\hat{H}^*(G,M)$ for some $kG$-module $M$.

There is another way of describing the product of $\hat{H}^*(G)$, in terms of $P_*$. Consider the differential graded algebra $\A=\Hom_L^*(P_*,P_*)$, which (in degree $n$) is given by
\[ \A^n=\prod_{j\in\mathbb{Z}} \Hom_L(P_{j+n},P_j), \]
and the differential $d:\A^n\longrightarrow \A^{n+1}$ is defined to be
\[ (df)_j=\partial\circ f_{j+1}-(-1)^n f_j\circ \partial. \]
Here $\partial$ denotes the differential of $P_*$. $\A$ is called the endomorphism dga of $P$. With this definition, the cocycles of $\A$ (of degree $n$) are exactly the chain transformations $P[n]\rightarrow P$, and two cocycles differ by a coboundary if and only if they are chain homotopic. Using standard arguments from homological algebra, one shows that the following map is an isomorphism of $k$-vector spaces:
\begin{equation} \label{tateiso}
\begin{split}
 H^n \A &\stackrel{\cong}{\longrightarrow} \Tate^n_L(k,k) \\
 \left[f\right] &\;\mapsto\;  [ \epsilon\circ f_0]
\end{split}
\end{equation}
Here $\epsilon:P_0\longrightarrow k$ is the augmentation. This isomorphism is compatible with the multiplicative structures. We will often write $\bar{a}$ for elements of the endomorphism dga; if $\bar{a}$ is a cocycle, then $a$ denotes the corresponding cohomology class.

\subsection{Hochschild Cohomology}
We now give a short review of Hochschild cohomology. Let $\Lambda$ be a graded algebra over the field $k$, and let $M$ be a graded $\Lambda$-$\Lambda$-bimodule, the elements of $k$ acting symmetrically. Define a cochain complex $C^{\bullet,*}(\Lambda,M)$ by
\[ C^{n,m}(\Lambda,M)=\Hom_k^m(\Lambda^{\otimes n},M), \]
with a differential $\delta$ of bidegree $(1,0)$ given by
\begin{multline*} (\delta\varphi)(\lambda_1,\dots,\lambda_{n+1})=(-1)^{m|\lambda_1|}\lambda_1\varphi(\lambda_2,\dots,\lambda_{n+1}) \\
+\sum_{i=1}^n (-1)^i\varphi(\lambda_1,\dots,\lambda_i\lambda_{i+1},\dots,\lambda_{n+1}) +(-1)^{n+1}\varphi(\lambda_1,\dots,\lambda_n)\lambda_{n+1}.
\end{multline*}
The Hochschild cohomology groups $\hochschild^{*,*}(\Lambda,M)$ are defined as the cohomology groups of that complex:
\[ \hochschild^{s,t}(\Lambda,M)=H^s(C^{*,t}(\Lambda,M)). \]
In particular, we can regard $M=\Lambda$ as a bimodule over itself; then one writes $\hochschild^{s,t}(\Lambda)=\hochschild^{s,t}(\Lambda,\Lambda)$. For example, an element of $\hochschild^{3,-1}(\Lambda)$ is represented by a family of $k$-linear maps
\[ m=\{m_{i,j,l}:\Lambda^i\otimes\Lambda^j\otimes \Lambda^l\longrightarrow
        \Lambda^{i+j+l-1}\}_{i,j,l\in\mathbb{Z}} \]
satisfying the cocycle relation
\[ (-1)^{|a|}a\cdot m(b,c,d)-m(ab,c,d)+m(a,bc,d)-m(a,b,cd)+m(a,b,c)\cdot d=0 \]
for all $a,b,c,d\in\Lambda$.

Whenever $X$ and $Y$ are $\Lambda$-$\Lambda$-bimodules, one has a cup product pairing
\[ \cup:\Hom_\Lambda(X,Y)\otimes \hochschild^{*,*}\Lambda\longrightarrow \Ext^{*,*}_\Lambda(X,Y). \]
Here $\Ext^{s,t}_\Lambda(X,Y)$ is defined to be $\Ext^{s}_\Lambda(X,Y[t])$. In particular, we have the map
\begin{eqnarray*}
 \hochschild^{3,-1} \hat{H}^*(G) &\longrightarrow& \Ext^{3,-1}_{\hat{H}^*(G)} (X,X) \\
   \phi & \mapsto & \Id_X\cup\,\phi
\end{eqnarray*}
for every $\hat{H}^*(G)$-module $X$. This is the map occurring in the statement of Theorem~\ref{bkstheorem}.

\subsection{The canonical element $\gamma$} \label{kanonischel}
We are now going to describe the construction of the element $\gamma$ mentioned in Theorem~\ref{bkstheorem}. More generally, we will construct an element $\gamma_\A\in\hochschild^{3,-1} H^*\A$ for every differential graded algebra $\A$ over $k$; then we can take $\A$ to be the endomorphism algebra of a complete projective resolution of $k$ as a trivial $kG$-module to get $\gamma_G\in \hochschild^{3,-1} \hat{H}^*(G)$.

For a dg-algebra $\A$ consider $H^*\A$ as a differential graded $k$-module with trivial differential. Then choose a morphism of dg-$k$-modules $f_1:H^*\A\longrightarrow \A$ of degree $0$ which induces the identity in cohomology. This is the same as choosing a representative in $\A$ for every class in $H^*\A$ in a $k$-linear way. For every two elements $x,y\in H^*\A$, $f_1(xy)-f_1(x)f_1(y)$ is null-homotopic; therefore, we can choose a morphism of graded modules
\[ f_2:H^*\A\otimes H^*\A\longrightarrow \A \]
of degree $-1$ such that for all $x,y\in H^*\A$ we have
\[ df_2(x,y)=f_1(xy)-f_1(x)f_1(y). \]
Then for all $a,b,c\in H^*\A$,
\begin{equation} \label{mkonstruktion}
 f_2(a,b)f_1(c)-f_2(a,bc)+f_2(ab,c)-(-1)^{|a|}f_1(a)f_2(b,c) 
\end{equation}
is a cocycle in $\A$, the cohomology class of which will be denoted by $m(a,b,c)$. This defines a map $m:(H^*\A)^{\otimes 3}\longrightarrow H^*\A$ of degree $-1$. An explicit computation shows that $m$ is a Hochschild cocycle, thereby representing a class $\gamma_A\in\hochschild^{3,-1} H^*\A$. This class is independent of the choices made.

\section{Computation of the canonical element}
From now on, let $k$ be a field of characteristic $2$. Let $t\geq 2$ be a power of $2$, and let $G=Q_{4t}$ be the group of generalized quaternions
\[ Q_{4t} = \left< g,h \,\mid\, g^t = h^2,\, ghg=h \right>. \]
We denote by $kG$ the group algebra of $G$ over $k$, and $F=kG$ denotes the free module of rank~$1$ over that algebra. In this section, we are going to explicitly compute a Hochschild cochain $m$ representing the canonical class $\gamma_G$.

\subsection{The class of a map}\label{sklasse}
We begin with an observation that will reduce the subsequent computations somewhat. Let us recall the construction of a representative of $\gamma_{G}$. First of all, we have to construct a projective resolution $P$, and we will actually find a minimal projective resolution. Then we have to choose a cycle selection-homomorphism
$ f_1:\hat{H}^*(G)\rightarrow \Hom_{kG}^*(P,P) $
such that any class $a$ is mapped to a representative $f_1(a)$. We can find a $k$-linear map
$ f_2:\hat{H}^*(G)\otimes\hat{H}^*(G)\rightarrow \Hom_{kG}^*(P,P) $
of degree $-1$ satisfying
$ df_2(a,b)=f_1(a)f_1(b)-f_1(ab) $
for all $a,b$. Finally, we are interested in terms of the form
\begin{equation}\label{kozykelterm} f_2(a,b)f_1(c)+f_2(a,bc)+f_2(ab,c)+f_1(a)f_2(b,c); \end{equation}
this is a cocycle in $\Hom_{kG}^*(P,P)$. In order to determine the class of this cocycle, it is enough to know the degree $0$ map of it (cf.~\eqref{tateiso}). This observation leads to the following definition.
\begin{definition}
For every $f\in \Hom^n_{kG}(P,P)$, i.e.,~a family of maps $f_j:P_{j+n}\rightarrow P_j$ ($j\in\mathbb{Z}$), not necessarily commuting with the differential, we denote by $\class(f)$ the class of the map 
$\epsilon \circ f_0: P_n\rightarrow k$ in $H^n \Hom_{kG}(P_*,k)=\hat{H}^*(G)$. 
\end{definition}
Note that the complex $\Hom_{kG}(P_*,k)$ has trivial differential; thus, every element in $\Hom_{kG}(P_*,k)$ and in particular $\epsilon\circ f_0$ is a cocycle.
The definition above gives a map
\[
\boxed{
\begin{aligned}
 \class:\Hom^n_{kG}(P,P) &\longrightarrow  \hat{H}^n(G) \\
  f & \;\mapsto\;  [\epsilon\circ f_0]
\end{aligned}
}
\]
\begin{lemma} \label{classlemma}
The map $\class$ has the following properties:
\begin{enumerate}
\item[(i)] If $f\in\Hom^n_{kG}(P,P)$ is a cocycle, then $\class(f)$ is the cohomology class of $f$; in particular $\class\circ f_1=\Id$.
\item[(ii)] The map $\class$ is $k$-linear.
\item[(iii)] If $\class(f_1)=\class(f_2)$ for some $f_1,f_2\in\Hom^n_{kG}(P,P)$, then $\class(f_1g)=\class(f_2g)$ for all $g\in\Hom^m_{kG}(P,P)$.
\item[(iv)] If $a\in\Hom^m_{kG}(P,P)$ is a cocycle and $f\in\Hom^n_{kG}(P,P)$ is arbitrary, then $\class(fa)=\class(f)\class(a)$.
\end{enumerate}
\end{lemma}
\begin{proof}
(i) follows from \eqref{tateiso}, (ii) holds by definition. \\
Ad (iii): If $\class(f_i)=0$, then $\epsilon\circ f_i=0$. This implies $\epsilon\circ f_i\circ g=0$, hence $\class(f_ig)=0$. For general $f_1,f_2$ note $\class(f_1-f_2)=0$; by what we just proved $\class((f_1-f_2)g)=0$ and therefore $\class(f_1g)=\class(f_2g)$.\\
Ad (iv): Choose a cocycle $h\in\Hom^n_{kG}(P,P)$ satisfying $\class(h)=\class(f)$. Then by (iii)
\[ \class(fa)=\class(ha)=\class(h)\class(a)=\class(f)\class(a). \qedhere \]
\end{proof}

The following corollary will simplify computations later on.
\begin{lemma} \label{vereinfachlemma}
The map $f_2$ can be chosen in such a way that $\class\circ f_2=0$.
\end{lemma}
\begin{proof}
Choose any $\tilde{f_2}$ (satisfying $d\tilde{f}_2(a,b)=f_1(a)f_1(b)-f_1(ab)$). Put $f_2=\tilde{f}_2-f_1\circ \class \circ \tilde{f}_2$. Since $df_1=0$, we get
\[ df_2(a,b)=d\tilde{f}_2(a,b)=f_1(a)f_1(b)-f_1(ab), \]
and from $\class\circ f_1=\Id$ follows that
\begin{equation*} \class\circ f_2=\class\circ \tilde{f}_2-\class\circ f_1\circ \class\circ \tilde{f}_2=0. \qedhere \end{equation*}
\end{proof}

Consider (\ref{kozykelterm}) with this simplified version of $f_2$. By applying $\class$, we get the term
\[ \class(f_2(a,b)f_1(c))+\class(f_2(a,bc))+\class(f_2(ab,c))+\class(f_1(a)f_2(b,c)) \]
This is the cohomology class of \eqref{kozykelterm}. Note that the individual terms $f_2(a,b)f_1(c)$, $f_2(a,bc)\dots$ will not be cocycles in general, but the map $\class$ assigns cohomology classes to them in such a way that the sum will be the class we are looking for.

By our choice of $f_2$ (such that $\class\circ f_2=0$), the first three terms in the sum vanish (note that $\class(f_2(a,b)f_1(c))=\class(f_2(a,b))c$ by Proposition~\ref{classlemma}.(iv)). Thus we are interested in terms of the form
$\class(f_1(a)f_2(b,c))$, where $a,b,c$ run through all elements of a $k$-basis of $\hat{H}^*(G)$.

\subsection{Generating cocycles and homotopies}
Now we start the actual computation of $\gamma$. We begin with the construction of a minimal projective resolution $P$ and some cocycles in the endomorphism dga of $P$. Let us define some elements of the group algebra $kG$ as follows. Put $a=g+1$, $b=h+1$ and $c=hg+1$. Furthermore, we write $N=\sum_{j\in G} j$ for the norm element. 
Here are some formulae we will frequently use:
\begin{align*}
 a^t &= b^2 = c^2 &
 a^{2t} & = b^4 = 0 \\
 ba &= ac = a+b+c & 
 N &= a^{2t-1}b \\
 c &= a+bg &
 gc &= a+b \\
 N &= ca^{2t-2}b = ca^{2t-1} &
 N &= a^{2t-1} + a^{2t-2}b + ca^{2t-2} \\
 ca^{t-1}b &= ca^{t-1} + a^{t-1}b
\end{align*}
Also note that $a^{2t-1}, a^{2t-2}$ and $a^{2t-4}$ lie in the center of $kQ_{4t}$. Now a $4$-periodic complete projective resolution of the trivial $kG$-module $k$ is given as follows (see \cite{CartEile}, Chapter XII \S 7):
\[
\xymatrix@C=15pt{
 \dots & P_0=F\ar[l]_-{N} && P_1=F^2 \ar[ll]_{\smatrix{ a & b }} && P_2=F^2 \ar[ll]_{\smatrix{a^{t-1} & c \\ b & a }} && P_3=F \ar[ll]_-{\smatrix{a \\ c}} && P_4=F \ar[ll]_-{N} & \dots \ar[l] } \]
Since the resolution is minimal, the differential of the complex $\Hom_{kG}(P_*,k)$ vanishes; therefore, we immediately get the well-known additive structure of $\hat{H}^*(G)$:
\begin{align*}
  \hat{H}^{4n}(G) \cong \hat{H}^{4n+3}(G) & \cong k, & \hat{H}^{4n+1}(G) \cong \hat{H}^{4n+2}(G) & \cong k^2. 
\end{align*}
Let us write $\bar{s}:P\rightarrow P[4]$ for the shift map, given by the identity map in every degree. This is an invertible cocycle; thus, multiplication by a suitable power of $s$ yields an isomorphism $\hat{H}^{4n+u}(G)\cong\hat{H}^{u}(G)$ for $u=0,1,2,3$ and $n\in\mathbb{Z}$. Now we are heading for explicit generators $x,y$ of $\hat{H}^1(G)\cong H^1\Hom_{kG}^*(P,P)$, which are represented by chain maps $\bar{x},\bar{y}:P[1]\rightarrow P$. By construction we have $P_1=F^2$ and $P_0=F$. We extend the two projections $P_1\rightarrow P_0$ to chain transformations $P[1]\rightarrow P$ as follows:
For $\bar{x}:P\rightarrow P[1]$ we take
\[
\xymatrix@C=35pt{
\dots & F\ar[l] \ar[d]^{a^{2t-2}b} & F^2 \ar[l]_{\smatrix{ a & b }} \ar[d]^{\smatrix{1 & 0}} & F^2 \ar[l]_{\smatrix{a^{t-1} & c \\ b & a }} \ar[d]^{\smatrix{a^{t-2} & 1 \\ 0 & g}} & F \ar[l]_{\smatrix{a \\ c}} \ar[d]^{\smatrix{1 \\ 1}} & F \ar[l]_{N} \ar[d]^{a^{2t-2}b} & \dots \ar[l] \\
\dots & F \ar[l] & F\ar[l]^{N} & F^2 \ar[l]^{\smatrix{ a & b }} & F^2 \ar[l]^{\smatrix{a^{t-1} & c \\ b & a }} & F \ar[l]^{\smatrix{a \\ c}}  & \dots \ar[l] } \]
and extend this $4$-periodically.
The $4$-periodic chain map $\bar{y}:P\rightarrow P[1]$ is defined as follows:
\[
\xymatrix@C=35pt{
\dots & F\ar[l] \ar[d]^{a^{2t-1}} & F^2 \ar[l]_{\smatrix{ a & b }} \ar[d]^{\smatrix{0 & 1}} & F^2 \ar[l]_{\smatrix{a^{t-1} & c \\ b & a }} \ar[d]^{\smatrix{0 & 1 \\ 1 & 0}} & F \ar[l]_{\smatrix{a \\ c}} \ar[d]^{\smatrix{0 \\ 1}} & F \ar[l]_{N} \ar[d]^{a^{2t-1}} & \dots \ar[l] \\
\dots & F \ar[l] & F\ar[l]^{N} & F^2 \ar[l]^{\smatrix{ a & b }} & F^2 \ar[l]^{\smatrix{a^{t-1} & c \\ b & a }} & F \ar[l]^{\smatrix{a \\ c}}  & \dots \ar[l] } \]
Since these cocycles are $4$-periodic, they commute with $\bar{s}$. Let us determine the pairwise products of these maps. We start with $\bar{x}\bar{y}$:
\[
\xymatrix@C=35pt{
\dots & F\ar[l] \ar[d]^{\smatrix{a^{2t-1} \\ a^{2t-1}}} & F^2 \ar[l]_{\smatrix{ a & b }} \ar[d]^{\smatrix{0 & a^{2t-2}b}} & F^2 \ar[l]_{\smatrix{a^{t-1} & c \\ b & a }} \ar[d]^{\smatrix{0 & 1}} & F \ar[l]_{\smatrix{a \\ c}} \ar[d]^{\smatrix{1 \\ g}} & F \ar[l]_{N} \ar[d]^{\smatrix{a^{2t-1} \\ a^{2t-1}}} & \dots \ar[l] \\
\dots & F^2 \ar[l] & F \ar[l]^{\smatrix{a \\ c}} & F\ar[l]^{N} & F^2 \ar[l]^{\smatrix{ a & b }} & F^2 \ar[l]^{\smatrix{a^{t-1} & c \\ b & a }} & \dots \ar[l] } \]
The product $\bar{y}\bar{x}$ is given as follows:
\[
\xymatrix@C=35pt{
\dots & F\ar[l] \ar[d]^{\smatrix{0 \\ a^{2t-2}b}} & F^2 \ar[l]_{\smatrix{ a & b }} \ar[d]^{\smatrix{a^{2t-1} & 0}} & F^2 \ar[l]_{\smatrix{a^{t-1} & c \\ b & a }} \ar[d]^{\smatrix{0 & g}} & F \ar[l]_{\smatrix{a \\ c}} \ar[d]^{\smatrix{1 \\ 1}} & F \ar[l]_{N} \ar[d]^{\smatrix{0 \\ a^{2t-2}b}} & \dots \ar[l] \\
\dots & F^2 \ar[l] & F \ar[l]^{\smatrix{a \\ c}} & F\ar[l]^{N} & F^2 \ar[l]^{\smatrix{ a & b }} & F^2 \ar[l]^{\smatrix{a^{t-1} & c \\ b & a }} & \dots \ar[l] } \]
Next, we compute $\bar{x}^2$: 
\[
\xymatrix@C=35pt{
\dots & F\ar[l] \ar[d]^{\smatrix{a^{2t-2}b \\ a^{2t-2}b}} & F^2 \ar[l]_{\smatrix{ a & b }} \ar[d]^{\smatrix{a^{2t-2}b & 0}} & F^2 \ar[l]_{\smatrix{a^{t-1} & c \\ b & a }} \ar[d]^{\smatrix{a^{t-2} & 1}} & F \ar[l]_{\smatrix{a \\ c}} \ar[d]^{\smatrix{a^{t-2}+1 \\ g}} & F \ar[l]_{N} \ar[d]^{\smatrix{a^{2t-2}b \\ a^{2t-2}b}} & \dots \ar[l] \\
\dots & F^2 \ar[l] & F \ar[l]^{\smatrix{a \\ c}} & F\ar[l]^{N} & F^2 \ar[l]^{\smatrix{ a & b }} & F^2 \ar[l]^{\smatrix{a^{t-1} & c \\ b & a }} & \dots \ar[l] } \]
And now $\bar{y}^2$:
\[
\xymatrix@C=35pt{
\dots & F\ar[l] \ar[d]^{\smatrix{0 \\ a^{2t-1}}} & F^2 \ar[l]_{\smatrix{ a & b }} \ar[d]^{\smatrix{0 & a^{2t-1}}} & F^2 \ar[l]_{\smatrix{a^{t-1} & c \\ b & a }} \ar[d]^{\smatrix{1 & 0}} & F \ar[l]_{\smatrix{a \\ c}} \ar[d]^{\smatrix{1 \\ 0}} & F \ar[l]_{N} \ar[d]^{\smatrix{0 \\ a^{2t-1}}} & \dots \ar[l] \\
\dots & F^2 \ar[l] & F \ar[l]^{\smatrix{a \\ c}} & F\ar[l]^{N} & F^2 \ar[l]^{\smatrix{ a & b }} & F^2 \ar[l]^{\smatrix{a^{t-1} & c \\ b & a }} & \dots \ar[l] } \]
In each of these cocycles, the map $P_2\rightarrow P_0$ determines the cohomology class by the isomorphism \eqref{tateiso}; in $k^2$, they correspond to $(0\,1), (0\,1), (\epsilon(a^{t-2})\,1)$ and $(1\,0)$, respectively. Hence $\hat{H}^2(G)$ is generated by $x^2$ and $y^2$, and we have $xy=yx$. Furthermore, we also see from this description that
\[ xy = \begin{cases} x^2+y^2 & \text{if $t=2$,} \\ x^2 & \text{otherwise.}  \end{cases} \]
But we will need explicit chain homotopies for all these relations later on, so let us start with the commutator relation $xy=yx$. Let $\bar{p}$ be the $4$-periodic null-homotopy for $\bar{x}\bar{y}+\bar{y}\bar{x}$ defined as follows:
\[
\xymatrix@C=35pt{
\dots & F\ar[l] \ar[dr]^(0.6){a^{2t-2}} & F^2 \ar[l]_{\smatrix{ a & b }} \ar[dr]^(0.6){0} & F^2 \ar[l]_{\smatrix{a^{t-1} & c \\ b & a }} \ar[dr]^(0.7){\smatrix{0 & 1 \\ 0 & 0}} & F \ar[l]_{\smatrix{a \\ c}} \ar[dr]^(0.7){\smatrix{0 \\ 1}} & F \ar[l]_{N} & \dots \ar[l] \\
\dots & F^2 \ar[l] & F \ar[l]^{\smatrix{a \\ c}} & F\ar[l]^{N} & F^2 \ar[l]^{\smatrix{ a & b }} & F^2 \ar[l]^{\smatrix{a^{t-1} & c \\ b & a }} & \dots \ar[l] } \]
Now let us compute $\bar{y}^3$:
\[
\xymatrix@C=35pt{
\dots & F \ar[l] \ar[d]^{0} & F\ar[l]_{N} \ar[d]^{\smatrix{a^{2t-1} \\ 0}} & F^2 \ar[l]_{\smatrix{ a & b }} \ar[d]^{\smatrix{0 & 0 \\ 0 & a^{2t-1}}}  & F^2 \ar[l]_{\smatrix{a^{t-1} & c \\ b & a }} \ar[d]^{\smatrix{a^{2t-1} & 0}}  & F \ar[l]_{\smatrix{a \\ c}} \ar[d]^{0} & \dots \ar[l] \\
\dots & F \ar[l]  & F^2 \ar[l]^{\smatrix{ a & b }} & F^2 \ar[l]^{\smatrix{a^{t-1} & c \\ b & a }} & F \ar[l]^{\smatrix{a \\ c}} & F\ar[l]^{N} & \dots \ar[l] } \]
Then we find a null-homotopy for that map in two steps: First consider the $4$-periodic extension of the map
\[
\xymatrix@C=20pt@R=40pt{
\dots & F \ar[l] \ar[drr]^(0.6){0} && F\ar[ll]_{N} \ar[drr]^(0.7){\smatrix{bh^{-1} \\ a^{t-1}h^{-1}}} && F^2 \ar[ll]_{\smatrix{ a & b }} \ar[drr]^(0.7){\!\!\smatrix{cg^{-1} \\ a^{t-1}}^{\!T}\! h^{-1}}  && F^2 \ar[ll]_{\smatrix{a^{t-1} & c \\ b & a }} \ar[drr]^(0.6){0}  & & F \ar[ll]_{\smatrix{a \\ c}} & \dots \ar[l] \\
\dots & F \ar[l]  && F^2 \ar[ll]^{\smatrix{ a & b }} && F^2 \ar[ll]^{\smatrix{a^{t-1} & c \\ b & a }} && F \ar[ll]^{\smatrix{a \\ c}} && F\ar[ll]^{N} & \dots \ar[l] } \]
and call it $\bar{w}'$. Note that this will not quite be a homotopy for $\bar{y}^3$, because it yields the wrong result in degrees $P_{4n+2}\rightarrow P_{4n-1}$ for all $n\in\mathbb{Z}$. But if we put
\[ P_{8n+j+3}\rightarrow P_{8n+j}: \bar{w}_{8n+j} = \begin{cases} \bar{w}'_{8n+j} & \text{if $j=0,1,2,3$,} \\ (\bar{w}'+\bar{y}^2)_{8n+j} & \text{if $j=4,5,6,7$,} \end{cases} \]
then we get an $8$-periodic null-homotopy for $\bar{y}^3$ which will be called $\bar{w}$ and satisfies $\bar{s}\bar{w} +\bar{w}\bar{s} = \bar{y}^2$.

\subsection{Computation for the quaternion group}
Due to the different multiplicative relation in $\hat{H}^*(G)$ we need to consider the cases $t=2$ and $t\geq 4$ separately. We start with $t=2$. In this case, the map
\[
\xymatrix@C=35pt{
\dots & F^2 \ar[l] \ar[dr]^(0.6){0} & F^2 \ar[l]_{\smatrix{a & c \\ b & a }} \ar[dr]^(0.6){0} & F \ar[l]_{\smatrix{a \\ c}} \ar[dr]^(0.6){0} & F \ar[l]_{N} \ar[dr]^(0.6){a^3+a^2+ab} & F^2 \ar[l]_{\smatrix{ a & b }} & \dots \ar[l] \\
\dots & F \ar[l]  & F\ar[l]^{N} & F^2 \ar[l]^{\smatrix{ a & b }} & F^2 \ar[l]^{\smatrix{a & c \\ b & a }} & F \ar[l]^{\smatrix{a \\ c}} & \dots \ar[l] } 
\]
can be extended (as we did with $\bar{w}$ above) to an $8$-periodic null-homotopy $\bar{r}$ for $\bar{x}^2+\bar{x}\bar{y}+\bar{y}^2$ satisfying $\bar{s}\bar{r} + \bar{r}\bar{s} = \bar{x}+\bar{y}$. Notice that $\bar{x}\bar{y}^2:P_3\rightarrow P_0$ is the identity map, which implies that $xy^2\neq 0\in\hat{H}^3(G)$. Gathering the results we obtained so far, we recover the known fact that 
\begin{align*}
\hat{H}^*(G) \cong k[x,y,s^{\pm 1}]/(x^2+y^2=xy,y^3=0). 
\end{align*}
Let us remark here that all monomials in $x$ and $y$ of degree bigger than $3$ vanish in this ring.

\begin{lemma}\label{pqreigenschaften}
Let $\alpha,\beta,\gamma$ be monomials in the (non-commutative) variables $\bar{x},\bar{y}$, and assume that the degree $|\beta|\geq 3$.
Then we have the following formulae:
\begin{align*}
 \class(\bar{p}\alpha) &= 0 & \class(\bar{r}\alpha) &= 0 &\class(\bar{w}\alpha) &= 0  \\ 
 \class(\bar{x}\bar{p}\alpha) &= xy\class(\alpha) & \class(\gamma\bar{r}\alpha)&=0 & \class(\gamma\bar{w}\alpha)&= 0 \\
 \class(\bar{y}\bar{p}\alpha) &= 0 \\
 \class(\bar{x}^2\bar{p}\alpha) &= x^2y\class(\alpha) \\
 \class(\bar{y}^2\bar{p}\alpha) &= 0 \\
 \class(\beta\bar{p}\alpha) &= 0
\end{align*}
\end{lemma}
\begin{proof}
By Proposition~\ref{classlemma}.(iii) we can assume that the degree of $\beta$ is at most $3$. Furthermore, we can assume $\alpha=1$ by Proposition~\ref{classlemma}.(iv). In order to determine $\class(\bar{a}\bar{w})$ for any given cocycle $\bar{a}$ of degree $n$, we consider the composition
\begin{equation*}
  P_{n+2} \xrightarrow{\bar{w}_n} P_n \xrightarrow{\bar{a}_0} P_0 \xrightarrow{\hspace{1pt}\epsilon} k 
\end{equation*}
as an element of $H^{n+2}\Hom_{kG}(P_*,k)$. 
Notice $\im(\bar{w}_n)\subset \ker(\epsilon)\cdot P_n$. Therefore, $\im(\bar{a}_0\circ\bar{w}_n)\subset \ker(\epsilon)\cdot P_0=\ker(\epsilon)$, hence  $\epsilon\circ \bar{a}_0\circ \bar{w}_n=0$. The same proof works for $\bar{r}$ instead of $\bar{w}$, so we are left with $\bar{p}$. 
For $\class(\bar{x}\bar{p})$ consider $\bar{x}\bar{p}$ in degree $0$, i.e.,~
\[
\begin{array}{ccccc}
P_2 &\stackrel{\bar{p}_1}{\longrightarrow}&P_1&\stackrel{\bar{x}_0}{\longrightarrow}& P_0 \\
&\smatrix{0 & 1 \\ 0 & 0}&&\smatrix{1 & 0}
\end{array}
\]
that is $\smatrix{0 & 1}:P_2\longrightarrow P_0$, which corresponds to $xy$. The remaining cases can be shown analogously.
\end{proof}

\begin{bemerkung}\label{keine4periode} Using $\class$, we can prove that there is no $4$-periodic null-homotopy for $\bar{x}^2+\bar{x}\bar{y}+\bar{y}^2$ as follows: Suppose there is a $4$-periodic null-homotopy; call it $\hat{r}$. Since $d(\hat{r}-\bar{r})=0$, $\bar{q}=\hat{r}-\bar{r}$ is a cocycle, representing some class $q$. By construction,  $\bar{s}\bar{r}=(\bar{r}+\bar{x}+\bar{y})\bar{s}$. Since $\hat{r}$ is $4$-periodic, we have
$\class(\bar{s}\bar{q})=\class(\bar{q}\bar{s})-\class((\bar{x}+\bar{y})\bar{s})=qs-(x+y)s$
by Proposition~\ref{classlemma}. On the other hand, $\class(\bar{s}\bar{q})=sq$, hence $(x+y)s=0$, a contradiction. In a similar way one shows that there is no $4$-periodic null-homotopy for\rem{?} $\bar{x}^3$.
\end{bemerkung}

As a next step, we are going to define the functions $f_1$ and $f_2$. A $k$-basis of $\hat{H}^*(G)$ is given by $\mathfrak{C}=\{s^i, xs^i, ys^i, x^2s^i, y^2s^i, x^2ys^i\mid i\in\mathbb{Z}\}$. Define the $k$-linear map $f_1$ on the basis $\mathfrak{C}$ by
\begin{eqnarray*}
f_1\,:\, \hat{H}^*(G) &\rightarrow&\Hom_{kG}^*(P,P) \\
x^\varepsilon y^\delta s^i &\mapsto& \bar{x}^\varepsilon \bar{y}^\delta \bar{s}^i
\end{eqnarray*}
for all  $i,\varepsilon,\delta\in\mathbb{Z}$ for which the expression on the left hand side lies in $\mathfrak{C}$. Let us define the set $\B=\{1,x,y,x^2,y^2,x^2y\}$. For all $b,c\in \mathcal{B}$ and $i,j\in\mathbb{Z}$ we have $f_1(bs^ics^j)=f_1(bc)\bar{s}^{i+j}$ and $f_1(bs^i)f_1(cs^i)=f_1(b)f_1(c)\bar{s}^{i+j}$, since $\bar{s}$ commutes with both $\bar{x}$ and $\bar{y}$. This implies that we can define $f_2$ on $\mathcal{B}\times\mathcal{B}$ and then extend it to $\mathfrak{C}\times\mathfrak{C}$ via
$f_2(bs^i,cs^j)=f_2(b,c)\bar{s}^{i+j}$. Now define $f_2$ on $\mathcal{B}\times\mathcal{B}$ as follows:
\begin{flushleft}
\begin{tabular}{cc|cccc|} 
 \multicolumn{2}{c}{\multirow{2}*{$f_2(b,c)$}}    &     \multicolumn{4}{c}{$c$}  \\
\multicolumn{1}{c}{}     &      &$1$&$x$         &$y$         &$x^2$         \\  \cline{2-6}
\multirow{6}*{$b$}    &$1$   &$0$&$0$         &$0$         &$0$             \\
    &$x$   &$0$& $0$ & $\bar{r}$& $\bar{x}\bar{r}+\bar{r}{y}+\bar{w}$         \\
    &$y$   &$0$& $\bar{p}+\bar{r}$ & $0$ & $\bar{p}\bar{x}+\bar{x}\bar{p}+\bar{x}\bar{y}$ \\
    &$x^2$ &$0$& $\bar{x}\bar{r}+\bar{r}\bar{y}+\bar{w}$ & $0$ & $\bar{x}\bar{r}\bar{x}+\bar{r}\bar{y}\bar{x}+\bar{w}\bar{x}$ \\
    &$y^2$ &$0$& $\bar{y}\bar{p}+\bar{y}\bar{r}+\bar{w}+\bar{p}\bar{x}+\bar{x}\bar{p}+\bar{x}\bar{y}$ & $\bar{w}$ & $\bar{y}^2\bar{r}+\bar{y}^2\bar{p}+\bar{w}\bar{x}+\bar{w}\bar{y}$ \\
    &$x^2y$&$0$& $\bar{x}^2\bar{p}+\bar{x}\bar{r}\bar{y}+\bar{r}\bar{y}^2+\bar{w}\bar{y}+\bar{x}^2\bar{y}$ & $\bar{r}\bar{y}^2+\bar{x}\bar{w}+\bar{y}\bar{w}$ & $\ast$\\  \cline{2-6}
\end{tabular}
\end{flushleft}
\begin{flushright}
\begin{tabular}{cc|cc|} 
 \multicolumn{2}{c}{}    &     \multicolumn{2}{c}{}  \\
 \multicolumn{1}{c}{} &    &$y^2$      &$x^2y$    \\  \cline{2-4}
\multirow{6}*{$b$}    &$1$    &$0$        &$0$       \\
    &$x$   & $\bar{r}\bar{y}+\bar{w}$ & $\bar{x}\bar{r}\bar{y}+\bar{r}\bar{y}^2+\bar{w}\bar{y}$ \\
    &$y$   & $\bar{w}$ & $\bar{y}\bar{r}\bar{y}+\bar{p}\bar{y}^2+\bar{x}\bar{w}+\bar{y}\bar{w}$ \\
    &$x^2$ & $\bar{r}\bar{y}^2+\bar{x}\bar{w}+\bar{y}\bar{w}$ & $\ast$ \\
    &$y^2$ & $\bar{w}\bar{y}$ & $\ast$ \\
    &$x^2y$& $\ast$ & $\ast$ \\  \cline{2-4}
\end{tabular}
\end{flushright}
Direct verification shows that $df_2(b,c)=f_1(bc)-f_1(b)f_1(c)$ for all $b,c$ for which $f_2$ is defined. Each $\ast$ can be replaced by a suitable polynomial expression in $\bar{x},\bar{y},\bar{p},\bar{r},\bar{w}$ such that $df_2(b,c)=f_1(bc)-f_1(b)f_1(c)$ holds for all $b,c$; as will turn out, it does not matter which choice we make here. Our $f_2$ will then already be simplified in the sense of Proposition~\ref{vereinfachlemma}, which is why some apparently unnecessary terms occur (e.g., the $\bar{x}\bar{y}$ in $f_2(y,x^2)$). Indeed, $\class\circ f_2=0$, as one can check using Proposition~\ref{pqreigenschaften}. 

As a final step, we need to investigate the term
\[ m(a,b,c)=\class(f_1(a)f_2(b,c)) \]
for all $a,b,c\in\mathfrak{C}$. Since $f_2(b,c)$ is $8$-periodic, we have
$m(as^{2h},bs^i,cs^j)=m(a,b,c)s^{2h+i+j}$ for all integers $h,i,j$ and $a,b,c\in\mathfrak{C}$. Therefore it is enough to consider all triples $(a,b,c)\in\bigl(\mathcal{B}\cup\mathcal{B} s\bigr)\times\mathcal{B}\times\mathcal{B}$.

Consider the case $a\in\mathcal{B}$. If $a=1$, then
$\class(f_1(a)f_2(b,c))=\class(f_2(b,c))=0$. If $a\in\{y^2,x^2y\}$, then $f_1(a)f_2(b,c)$ is a sum of terms $\beta\bar{p}\alpha$, $\beta\bar{r}\alpha$, $\beta\bar{w}\alpha$ and $\beta\bar{x}\bar{y}\alpha$, where $\alpha$ and $\beta$ are monomials in $\bar{x}$ and $\bar{y}$, and the degree of $\beta$ is at least $2$ and $\beta\neq\bar{x}^2$. Hence $\class(f_1(a)f_2(b,c))=0$ by Proposition~\ref{pqreigenschaften}.

Next, consider $a=x$. By Proposition~\ref{pqreigenschaften} we get $\class(\bar{x}f_2(b,c))$ from $f_2(b,c)$ by the following rule: Put an $\bar{x}$ in front of all monomials in $\bar{x}$ and $\bar{y}$. Then remove all summands containing $\bar{p}$, $\bar{r}$ or $\bar{w}$, except those beginning with $\bar{p}$, $\bar{x}\bar{p}$ or $\bar{y}\bar{p}$, where we replace the $\bar{p}$ by $xy$, and $\bar{x}\bar{p}$ and $\bar{y}\bar{p}$ by $x^2y$. Finally, replace all $\bar{x}$ and $\bar{y}$ by $x$ and $y$, respectively. Using this procedure, we get the following table for $\class(\bar{x}f_2(b,c))$:
\begin{center}
\begin{tabular}{cc|cccccc|} 
\multicolumn{2}{c}{\multirow{2}*{$\class(\bar{x}f_2(b,c))$}}  & \multicolumn{6}{c}{$c$}     \\
\multicolumn{1}{c}{}    &      &$1$&$x$         &$y$         &$x^2$       &$y^2$      &$x^2y$    \\  \cline{2-8}
\multirow{6}*{$b$}    &$1$   &$0$&$0$         &$0$         &$0$         &$0$        &$0$   \\
    &$x$   &$0$&$0$         &$0$         &$0$         &$0$        &$0$       \\
    &$y$   &$0$&$xy$       &$0$         &$xyx+x^2y+x^2y$   &$0$        &$\ast$ \\
    &$x^2$ &$0$&$0$         &$0$         &$\ast$         &$\ast$        &$\ast$       \\
    &$y^2$ &$0$&$x^2y+xyx+x^2y+x^2y$  &$0$     &$\ast$       &$\ast$        &$\ast$    \\
    &$x^2y$&$0$&$\ast$         &$\ast$         &$\ast$         &$\ast$        &$\ast$       \\  \cline{2-8}
\end{tabular}
\end{center}
Here each $\ast$ stands for some homogeneous polynomial in $x,y$ of degree at least $4$. Almost all these expressions vanish, the only remaining terms are
\begin{align*}
 m(x,y,x)   &= xy, \\
 m(x,y,x^2) &= x^2y.
\end{align*}
For the case $a=y$ we use a similar method resulting from Proposition~\ref{pqreigenschaften}, and we end up with $m(y,b,c)=0$ for all $b,c\in\B$. Finally,
for $a=x^2$ we find that the only non-zero term is $m(x^2,y,x)=x^2y$.

The case $a\in\mathcal{B}s$ is slightly more difficult. Consider the map
\[ h(b,c)=\Bar{s}f_2(b,c)\Bar{s}^{-1}-f_2(b,c), \]
measuring how far away $f_2$ is from $4$-periodicity.
From the equations
\begin{eqnarray*}
\bar{s}\bar{p}\bar{s}^{-1}&=&\bar{p} \\
\bar{s}\bar{r}\bar{s}^{-1}&=&\bar{r}+\bar{x}+\bar{y} \\
\bar{s}\bar{w}\bar{s}^{-1}&=&\bar{w}+\bar{y}^2  
\end{eqnarray*}
we get the following table for $h$:
\begin{flushleft}
\begin{tabular}{cc|ccc|} 
\multicolumn{2}{c}{\multirow{2}*{$h(b,c)$}}    & \multicolumn{3}{c}{$c$}   \\
\multicolumn{1}{c}{} &  &$1$&$x$         &$y$                 \\  \cline{2-5}
\multirow{6}*{$b$}    &$1$   &$0$&$0$         &$0$                     \\
    &$x$   &$0$& $0$ & $\bar{x}+\bar{y}$         \\
    &$y$   &$0$& $\bar{x}+\bar{y}$ & $0$  \\
    &$x^2$ &$0$& $\bar{x}(\bar{x}+\bar{y})+(\bar{x}+\bar{y})\bar{y}+\bar{y}^2$ & $0$ \\
    &$y^2$ &$0$& $\bar{y}(\bar{x}+\bar{y})+\bar{y}^2$ & $\bar{y}^2$  \\
    &$x^2y$&$0$& $\bar{x}(\bar{x}+\bar{y})\bar{y}+(\bar{x}+\bar{y})\bar{y}^2+\bar{y}^2\bar{y}$ & $(\bar{x}+\bar{y})\bar{y}^2+\bar{x}\bar{y}^2+\bar{y}\bar{y}^2$ \\  \cline{2-5}
\end{tabular}
\end{flushleft}
\begin{flushright}
\begin{tabular}{cc|ccc|}
\multicolumn{2}{c}{\multirow{2}*{}}    & \multicolumn{3}{c}{}   \\
\multicolumn{1}{c}{}  &     & $x^2$ &$y^2$      &$x^2y$    \\  \cline{2-5}
\multirow{6}*{$b$}    &$1$  & $0$    &$0$        &$0$       \\
    &$x$   & $\bar{x}(\bar{x}+\bar{y})+(\bar{x}+\bar{y})\bar{y}+\bar{y}^2$ & $(\bar{x}+\bar{y})\bar{y}+\bar{y}^2$ & $\bar{x}(\bar{x}+\bar{y})\bar{y}+(\bar{x}+\bar{y})\bar{y}^2+\bar{y}^2\bar{y}$ \\
    &$y$   & $0$ & $\bar{y}^2$ & $\bar{y}(\bar{x}+\bar{y})\bar{y}+\bar{x}\bar{y}^2+\bar{y}\bar{y}^2$ \\
    &$x^2$ & $\bar{x}(\bar{x}+\bar{y})\bar{x}+(\bar{x}+\bar{y})\bar{y}\bar{x}+\bar{y}^2\bar{x}$ & $(\bar{x}+\bar{y})\bar{y}^2+\bar{x}\bar{y}^2+\bar{y}\bar{y}^2$ & $\ast$ \\
    &$y^2$ &  $\bar{y}^2(\bar{x}+\bar{y})+\bar{y}^2\bar{x}+\bar{y}^2\bar{y}$ & $\bar{y}^2\bar{y}$ & $\ast$ \\
    &$x^2y$&$\ast$ & $\ast$ & $\ast$ \\  \cline{2-5}
\end{tabular}
\end{flushright}
where $\ast$ denotes certain homogeneous polynomials in $\bar{x}$ and $\bar{y}$ of degree at least $4$. Applying $\class$ to this table and using relations in $\hat{H}^*(G)$, we get
\begin{center}
\begin{tabular}{cc|cccccc|} 
\multicolumn{2}{c}{\multirow{2}*{$\class(h(b,c))$}}  &  \multicolumn{6}{c}{$c$} \\
\multicolumn{1}{c}{}   &      &$1$&$x$         &$y$         &$x^2$  &$y^2$      &$x^2y$        \\ \cline{2-8}
\multirow{6}*{$b$}    &$1$   &$0$&$0$         &$0$         &$0$  & $0$       &$0$          \\
    &$x$   &$0$&$0$  &$x+y$ & $x^2$ & $x^2+y^2$&$x^2y$ \\
    &$y$   &$0$&$x+y$&$0$         &$0$  & $y^2$       &$0$   \\
    &$x^2$ &$0$&$x^2$       &$0$         &$0$  &$0$   &$0$      \\
    &$y^2$ &$0$&$x^2+y^2$ &$y^2$       &$0$  &$0$   &$0$      \\
    &$x^2y$&$0$&$x^2y$     &$0$    &$0$   &$0$   &$0$     \\  \cline{2-8}
\end{tabular}
\end{center}

By definition of $h$ we have $h(b,c)\bar{s}=\bar{s}f_2(b,c)-f_2(b,c)\bar{s}$, hence
\[ \class(h(b,c))s=\class(\Bar{s}f_2(b,c))-\underbrace{\class(f_2(b,c))}_{0} s =m(s,b,c). \]
Therefore, this table shows the values $m(s,b,c)$ with $b,c\in\mathcal{B}$. On the other hand, we know that $m$ is a Hochschild-cocycle, in particular for all $a,b,c\in\mathcal{B}$
\[ a\, m(s,b,c)+m(as,b,c)+m(a,sb,c)+m(a,s,bc)+m(a,s,b)c =0. \]
Using $m(a,s,b)c=m(a,1,b)sc=0$, $m(a,s,bc)=m(a,1,bc)s=0$ and $m(a,sb,c)=m(a,b,c)s$, we get
\begin{equation}\label{mgleichung} m(as,b,c)=a\, m(s,b,c)+m(a,b,c)s \end{equation}
We know the right hand side for all $a,b,c\in\mathcal{B}$. Gathering all results, we get the following theorem.
\begin{satz}\label{mquat8}
The canonical element $\gamma_G$ is represented by the Hochschild cocycle $m$ which is given by the formula
\begin{align*}
 m(x,y,x)   &= xy, \\
 m(x,y,x^2) &= x^2y, \\
 m(x^2,y,x) &= x^2y, \\
 m(a,b,c) &= 0 &&\text{for all other $a,b,c\in \B$,} \\
 m(sa,b,c) &= sm(a,b,c) + a \class(h(b,c)) &&\text{where $\class(h(b,c))$ is given by the table above,} \\
 m(s^{2i}a,s^jb,s^lc) &= s^{2i+j+l} m(a,b,c).
\end{align*}
The element $\gamma\in \hochschild^{3,-1} \hat{H}^*(G)$ represented by $m$ is non-trivial.
\end{satz}
\begin{proof}
It remains to prove the non-triviality of $\gamma$. Assume $m=\delta g$ for some Hochschild $(2,-1)$-cochain $g$. Then,
\[ m(a,b,c)=(\delta g)(a,b,c)=a\, g(b,c)+g(ab,c)+g(a,bc)+g(a,b)c \]
for all $a,b,c$. In particular,
\begin{align*}
0=m(y,x,y)&=yg(x,y)+g(yx,y)+g(y,xy)+g(y,x)y  \\
0=m(x,y,y)&=xg(y,y)+g(xy,y)+g(x,y^2)+g(x,y)y \\
0=m(y,y,x)&=yg(y,x)+g(y^2,x)+g(y,yx)+g(y,y)x \\
0=m(x,x,x)&=xg(x,x)+g(x^2,x)+g(x,x^2)+g(x,x)x \\
xy=m(x,y,x)&=xg(y,x)+g(xy,x)+g(x,yx)+g(x,y)x
\end{align*}
Adding up these equations we get (using $x^2+y^2=xy$)
\[ xy=x\cdot (g(x,y)+g(y,x)). \]
This implies $g(x,y)+g(y,x)=y$. On the other hand, interchanging the roles of $x$ and $y$ we get $g(x,y)+g(y,x)=x$, a contradiction.
\end{proof}

\subsection{Computation for the generalized quaternion group}
From now on, we assume that $t\geq 4$. Then there is an $8$-periodic null-homotopy $\bar{v}$ for $\bar{x}^2+\bar{x}\bar{y}$, partially given by
\[
\xymatrix@C=35pt{
\dots & F^2 \ar[l] \ar[dr]^(0.6){0} & F^2 \ar[l]_{\smatrix{a^{t-1} & c \\ b & a }} \ar[dr]^(0.7){\smatrix{a^{t-3} & 0 \\ 0 & 0}} & F \ar[l]_{\smatrix{a \\ c}} \ar[dr]^(0.6){0} & F \ar[l]_{N} \ar[dr]^(0.6){u} & F^2 \ar[l]_{\smatrix{ a & b }} & \dots \ar[l] \\
\dots & F \ar[l]  & F\ar[l]^{N} & F^2 \ar[l]^{\smatrix{ a & b }} & F^2 \ar[l]^{\smatrix{a^{t-1} & c \\ b & a }} & F \ar[l]^{\smatrix{a \\ c}} & \dots \ar[l] } \]
satisfying $\bar{s}\bar{v} +\bar{v}\bar{s} = \bar{x}$. Here we write $u=ca^{2t-2}+ba^{2t-3}$ and need to prove
\begin{align*}
au &= a^{2t-2}b + a^{2t-1}, &
cu &= a^{2t-2}b + a^{2t-1}, \\
ua &= a^{2t-2}b + N, &
ub &= a^{2t-2}b.
\end{align*}
For instance, to prove the first formula, note that
\[ au + aca^{2t-2} = aba^{2t-3} = a^{2t-3}ba = a^{2t-3}ac = ca^{2t-2} = (a+b+ac)a^{2t-2}. \]
The other formulae can be proved similarly. 

Again one verifies that $x^2y\neq 0$, so that we recover the well-known structure of $\hat{H}^*(G)$ to be
\[ \hat{H}^*(G) \cong k[x,y,s^{\pm 1}] / (y^3, x^2+xy). \]
Using the variable $z=x+y$, we obtain the isomorphism 
\[ \hat{H}^*(G) \cong k[x,z,s^{\pm 1}] / (xz,x^3+z^3). \]
In the following, we will frequently switch between these two descriptions.

\begin{lemma}
We have the following formulae.
\begin{align*}
 \class(\bar{p}\alpha) &= 0 & \class(\bar{v}\alpha) &= 0 &\class(\bar{w}\alpha) &= 0  \\ 
 \class(\bar{x}\bar{p}\alpha) &= x^2\class(\alpha) & \class(\gamma\bar{v}\alpha)&=0 & \class(\gamma\bar{w}\alpha)&= 0 \\
 \class(\bar{y}\bar{p}\alpha) &= 0 \\
 \class(\bar{x}^2\bar{p}\alpha) &= x^2y\class(\alpha) \\
 \class(\bar{y}^2\bar{p}\alpha) &= 0 \\
 \class(\beta\bar{p}\alpha) &= 0
\end{align*}
for any $\alpha,\beta,\gamma$ monomials in $\bar{x},\bar{y}$ with $|\beta|\geq 3$.
\end{lemma}
We omit the straightforward proof and turn to the definition of the maps $f_1$ and $f_2$. As before let $\B=\{1,x,y,x^2,y^2,x^2y\}$; we define $f_1$ as
\begin{align*}
 f_1(s^i x^a y^b) &= \bar{s}^i \bar{x}^a\bar{y}^b
\end{align*}
for all $a,b,i\in\mathbb{Z}$ for which $x^ay^b$ lies in $\B$.
Now we define $f_2$ on $\B\times \B$ as follows:
\begin{flushleft}
\begin{tabular}{cc|cccc|} 
 \multicolumn{2}{c}{\multirow{2}*{$f_2(b,c)$}}    &     \multicolumn{4}{c}{$c$}  \\
\multicolumn{1}{c}{}     &      &$1$&$x$         &$y$         &$x^2$         \\  \cline{2-6}
\multirow{6}*{$b$}    &$1$   & $0$ & $0$         & $0$         & $0$             \\
    &$x$   &$0$&$0$         &$\bar{v}$&$\bar{x}\bar{v}$       \\
    &$y$   &$0$&$\bar{p}+\bar{v}$&$0$         &$\bar{p}\bar{x}+\bar{x}\bar{p}+\bar{x}^2$ \\
    &$x^2$ &$0$&$\bar{x}\bar{v}$ & $0$ & $\bar{x}^2\bar{v}+\bar{x}\bar{v}\bar{y}+\bar{v}\bar{y}^2+\bar{x}\bar{w}$     \\   
    &$y^2$ &$0$& $\bar{y}\bar{p}+\bar{p}\bar{y}+\bar{v}\bar{y}$ & $\bar{w}$ &  $\bar{y}^2\bar{v}+\bar{y}^2\bar{p}+\bar{w}\bar{x}$   \\
    &$x^2y$ & $0$ & $\bar{x}^2\bar{p}+\bar{x}\bar{v}\bar{y}+\bar{v}\bar{y}^2+\bar{x}\bar{w}+\bar{x}^2\bar{y}$ & $\bar{v}\bar{y}^2+\bar{x}\bar{w}$    & $\bar{x}^2\bar{p}\bar{x}+\bar{x}\bar{v}\bar{y}\bar{x}+\bar{v}\bar{y}^2\bar{x}+\bar{x}\bar{w}\bar{x}$ \\  \cline{2-6}
\end{tabular}
\end{flushleft}
\begin{flushright}
\begin{tabular}{cc|cc|} 
 \multicolumn{2}{c}{}    &     \multicolumn{2}{c}{}  \\
 \multicolumn{1}{c}{} &    &$y^2$      &$x^2y$    \\  \cline{2-4}
\multirow{6}*{$b$}    &$1$    &$0$        &$0$       \\
    &$x$   & $\bar{v}\bar{y}$ & $\bar{x}\bar{v}\bar{y}+\bar{v}\bar{y}^2+\bar{x}\bar{w}$    \\
    &$y$   & $\bar{w}$ & $\bar{y}\bar{v}\bar{y}+\bar{p}\bar{y}^2+\bar{x}\bar{w}$   \\
    &$x^2$ &$\bar{v}\bar{y}^2+\bar{x}\bar{w}$ & $\bar{x}^2\bar{v}\bar{y}+\bar{x}\bar{v}\bar{y}^2+\bar{x}^2\bar{w}$    \\
    &$y^2$ & $\bar{w}\bar{y}$ & $\bar{y}^2\bar{v}\bar{y}+\bar{y}^2\bar{p}\bar{y}+\bar{w}\bar{x}\bar{y}$ \\
   &$x^2y$ & $\bar{x}^2\bar{w}$ & $\bar{x}^2\bar{y}\bar{v}\bar{y}+\bar{x}^2\bar{p}\bar{y}^2+\bar{x}^3\bar{w}$ \\  \cline{2-4}
\end{tabular}
\end{flushright}
Also put $f_2(s^ia,s^jb) = f_2(a,b)\bar{s}^{i+j}$ for all $i,j\in\mathbb{Z}$ and $a,b\in\B$. This function is chosen in such a way that $\class(f_2(a,b))=0$ for all $a,b\in\B$. One verifies that
\begin{align*}
m(x,y,x) &= x^2, \\
m(x^2,y,x)&=x^2y, \\
m(x,y,x^2) &= x^2y
\end{align*}
and $m$ vanishes on all other triples $(a,b,c)\in\B^{\times 3}$. Let us define $m'$ as follows:
\begin{align} \label{definitionofmprime}
 m'(s^ia,s^jb,s^kc) &= s^{i+j+k} m(a,b,c) &\text{for all $a,b,c\in \B$},
\end{align}
and define $h(a,b) = \bar{s}f_2(a,b)\bar{s}^{-1}-f_2(a,b)$. Then $\class(h(b,c))$ is given by the following table:
\begin{center}
\begin{tabular}{cc|cccccc|} 
 \multicolumn{2}{c}{\multirow{2}*{$\class(h(b,c))$}}    &     \multicolumn{4}{c}{$c$}  \\
\multicolumn{1}{c}{}     &      &$1$&$x$         &$y$         &$x^2$         &$y^2$      &$x^2y$ \\  \cline{2-8}
\multirow{6}*{$b$}    &$1$   & $0$ & $0$         & $0$         & $0$               &$0$        &$0$   \\
    &$x$   &$0$&$0$         &$x$&$x^2$  & $x^2$ & $x^2y$     \\
    &$y$   &$0$&$x$&$0$         & $0$ & $y^2$ & $0$\\
    &$x^2$ &$0$&$x^2$ & $0$ & $0$   & $0$ & $0$    \\   
    &$y^2$ &$0$& $x^2$ & $y^2$ &  $0$  & $0$ & $0$   \\
    &$x^2y$ & $0$ & $x^2y$ & $0$    & $0$ & $0$ & $0$  \\  \cline{2-8}
\end{tabular}
\end{center}
So we get the following explicit description of $m$:
\begin{satz}\label{mquat16}
The canonical element $\gamma_G$ is represented by the Hochschild cocycle $m$ which is given by the formula
\begin{align*}
  m(x,y,x) &= x^2, \\
  m(x^2,y,x) &= x^2y, \\
  m(x,y,x^2) &= x^2y, \\
  m(a,b,c) &= 0 &&\text{for all other $a,b,c\in \B$,} \\
  m(sa,b,c) &= sm(a,b,c) + s a \class(h(b,c)) &&\text{where $\class(h(b,c))$ is given by the table above,} \\
  m(s^{2i}a,s^jb,s^lc) &= s^{2i+j+l} m(a,b,c).
\end{align*}
The element $\gamma\in \hochschild^{3,-1} \hat{H}^*(G)$ represented by $m$ is non-trivial.
\end{satz}
\begin{proof}
It remains to prove the non-triviality of $\gamma$. Suppose that $m$ is a Hochschild coboundary; then $m=\delta g$ for some $g:\Lambda^{\otimes 2}\rightarrow\Lambda[-1]$. Adding up the equations
\begin{align*}
 x^3 = m(x,z,x^2) &= xg(z,x^2)+g(x,z)x^2 \\
 0  = m(x^2,x,z) &= x^2g(x,z) + g(x^3,z) + g(x^2,x)z \\
 0  = m(z,x^2,x) &= zg(x^2,x) + g(z,x^3) + g(z,x^2)x \\
 0  = m(z,z^2,z) &= zg(z^2,z) + g(z^3,z) + g(z,z^3) + g(z,z^2)z \\
 0 = zm(z,z,z)   &= z^2g(z,z) + zg(z^2,z) + zg(z,z^2) + zg(z,z)z
\end{align*}
and simplifying, we get the contradiction $x^3=0$.
\end{proof}

\section{Realizability of modules}
\subsection{Massey products}
There is a strong connection between the canonical class $\gamma$ and triple Massey products over $\hat{H}^*(G)$. This has already been noted in \cite{bks}, Lemma 5.14, and we will generalize this fact to Massey products of matrices (as introduced by May, \cite{May}). We start with some notation. Let $\Lambda$ be a graded $k$-algebra, and suppose that $I$ is a graded set, that is, a set together with a function $|\cdot |:I\rightarrow \mathbb{Z}$. For every such set, we define $I[n]$ to be the shifted graded set given by the same set with new grading $|i|_{[n]}=|i|+n$ for all $i\in I$. We denote by $\Lambda^I$ the shifted free $\Lambda$-module 
\[ \Lambda^I = \bigoplus_{i\in I} \Lambda[|i|]. \]
Then $\Lambda^I [n] = \Lambda^{I[n]}$. If $J$ is another graded set, we can consider morphisms $f:\Lambda^J\rightarrow \Lambda^I$. Every such map can be represented by a (possibly infinite) matrix $(f_{i,j})_{i\in I,j\in J}$ with $|f_{i,j}|=|i|-|j|$. Such a matrix is column-finite, that is, for every $j$ there are only finitely many non-zero $f_{i,j}$'s. Let us denote by $\Lambda^{I,J}$ the set of such matrices. Every such yields a map $f:\Lambda^J \rightarrow\Lambda^I$. 

A triple of matrices $(A,B,C)$ will be called \emph{composable} if there are graded sets $I,J,K,L$ with $A\in \Lambda^{I,J}, B\in\Lambda^{J,K},C\in\Lambda^{K,L}$. Every morphism $m:\Lambda^{\otimes 3}\rightarrow \Lambda[-1]$ can be extended to the module of all composable triples by putting
\[ m(A,B,C)\in \Lambda^{I[-1],L}: \quad m(A,B,C)_{i[-1],l} = \sum_{j\in J}\sum_{k\in K} m(a_{ij},b_{jk},c_{kl}). \]
From now on we assume $\Lambda = H^*\A \cong \hat{H}^*(G)$, where $\A$ is the endomorphism-dgA of some projective resolution of the trivial $kG$-module $k$. Also let $m:\Lambda^{\otimes 3}\rightarrow \Lambda[-1]$ be some Hochschild cocycle representing the canonical element $\gamma\in\hochschild^{3,-1} \hat{H}^*(G)$. Recall that (see e.g.~\cite{May}) for every composable triple of matrices $(A,B,C)$ with $AB=0$ and $BC=0$ the triple matric Massey product $\left<A,B,C\right>$ is defined and a coset of $A\cdot \Lambda^{J[-1],L} + \Lambda^{I[-1],K}\cdot C$. Notice that there is no obstruction to generalizing May's definition to infinite matrices.

\begin{lemma}\label{mandmasseyproducts}
 For every composable triple $(A,B,C)$ with $AB=0$ and $BC=0$ we have that $m(A,B,C) \in \left<A,B,C\right>$.
\end{lemma}
\begin{proof}
 We have
\begin{align*}
 m(A,B,C) &= f_1(A)f_2(B,C) + f_2(AB,C) + f_2(A,BC) + f_2(A,B)f_1(A) \\
          &= f_1(A)f_2(B,C) + f_2(A,B)f_1(C),
\end{align*}
and the last term represents one element of the Massey product.
\end{proof}

A triple $(A,B,C)$ will be called \emph{exact} if it is composable and the sequence
\[ \Lambda^I \xleftarrow{A} \Lambda^J \xleftarrow{B} \Lambda^K \xleftarrow{C} \Lambda^L \]
is exact. 

\begin{lemma}\label{propexactindeterminacy}
Let $A\in\Lambda^{I,J}$ be any matrix, and define $M = \coker A$. Then the following are equivalent:
\begin{itemize}
 \item[(i)] The module $M$ is a direct summand of a realizable module.
 \item[(ii)] For every composable triple $(A,B,C)$ with $AB=0$ and $BC=0$, we have that $0\in\left<A,B,C\right>$.
 \item[(iii)] For some exact triple $(A,B,C)$ we have $0\in\left<A,B,C\right>$.
\end{itemize}
\end{lemma}
\begin{proof}
For $(i)\Rightarrow (ii)$, let $M$ be a direct summand of $H^*N$, where $N$ is some dg-$\A$-module. Then there are maps $M\xrightarrow{i} H^*N \xrightarrow{r} M$ with $ri=\id_M$. Let $\pi:\Lambda^I \rightarrow M$ be the projection map, and put $W = i\pi$. Then $WA=0$, so that $\left<W,A,B\right>$ is defined, and the juggling formula (see Corollary 3.2.(iii) of \cite{May}) yields $W\left<A,B,C\right> = \left<W,A,B\right>C$ as cosets of $W \Lambda^{I[-1],K} C$. Let $E:\Lambda^K\rightarrow H^*N[-1]$ be some element in $\left<W,A,B\right>$. Since $\Lambda^K$ is free, we know that the composition $r\circ E$ lifts as $\Lambda^K \xrightarrow{S} \Lambda^{I[-1]} \xrightarrow{\pi} M[-1]$ for some matrix $S$. But then
\begin{align*}
  \pi S C = r E C \in r \left<W,A,B\right> C = r W \left<A,B,C\right> = \pi \left<A,B,C\right>. 
\end{align*}
This means that there is some matrix $T$ such that $AT + SC \in \left<A,B,C\right>$, which implies $0\in\left<A,B,C\right>.$

The implication $(ii)\Rightarrow (iii)$ is obvious. For $(iii)\Rightarrow (i)$, note that 
\[ M \leftarrow \Lambda^I \xleftarrow{A} \Lambda^J \xleftarrow{B} \Lambda^K \xleftarrow{C} \Lambda^L \]
is the beginning of a (shifted) free resolution of $M$. We have $m(A,B,C)\in \Lambda^{I[-1],L}$, and a representative of $\gamma\cup\id_M\in\Tate^{3,-1}_\Lambda(M,M)$ is given by the composition
\[ g:\Lambda^L \xrightarrow{m(A,B,C)} \Lambda^{I[-1]} \rightarrow (\coker A)[-1]=M[-1]. \]
By assumption and Proposition~\ref{mandmasseyproducts} $m(A,B,C) = AX+YC$ for some matrices $X$ and $Y$, so that this composition equals
\[ \Lambda^L \xrightarrow{C} \Lambda^K \xrightarrow{Y} \Lambda^{I[-1]} \rightarrow M[-1], \]
which in turn says that $g$ is the coboundary of $\Lambda^K \xrightarrow{Y} \Lambda^{I[-1]} \rightarrow M[-1]$; hence $\gamma\cup\id_M=0$. By Theorem~1.1 of \cite{bks}, $M$ is a direct summand of some realizable module.
\end{proof}

\subsection{The group of quaternions}
Let $G=Q_8$. We shall make use of one of the implications of Proposition~\ref{propexactindeterminacy} to prove the existence of a $\hat{H}^*G$-module which detects the non-triviality of $\gamma_G$:
\begin{satz}\label{mquat8module}
 The cokernel of the map
\[ \Lambda[-1]\oplus \Lambda[-1] \xrightarrow{\begin{pmatrix} y & x+y \\ x & y \end{pmatrix}} \Lambda \oplus \Lambda \]
is not a direct summand of a realizable $\hat{H}^*G$-module.
\end{satz}
\begin{proof}
Let $A=\left(\begin{smallmatrix} y & x+y \\ x & y \end{smallmatrix}\right)$; then $A^2=0$ and therefore the Massey product $\left<A,A,A\right>$ is defined. We claim that it does not contain $0$. An explicit calculation using the description of $m$ given in Theorem~\ref{mquat8} yields
\[ m(A,A,A) = \begin{pmatrix} x^2 & 0 \\ x^2 & x^2 \end{pmatrix}. \]
Let us denote the latter matrix by $B$, then by Proposition~\ref{propexactindeterminacy} we need to prove that $B$ is not of the form $B = A Q + R A$ for some $2\times 2$-matrices $Q$ and $R$. To do so, define $D=\smatrix{x & y \\ x+y & x}$; then $AD=DA=0$. If we denote by $\tr$ the trace of a matrix, then we have
\[ \tr(B D)=\tr(A Q D)+\tr(R A  D)
 =\tr(Q  D A)+\tr(R A  D)=0 \]
(note that these computations take place in a commutative ring). But
\[ \tr(B D)=\tr \begin{pmatrix} 0 & \ast \\ \ast & x^2y \end{pmatrix} = x^2y\neq 0, \]
a contradiction.
\end{proof}
\begin{remark}
 The triple $(A,A,A)$ is actually exact, but we do not need this.
\end{remark}

In order to construct a module which is not a direct summand of a realizable one, it is often enough to consider 'ordinary' Massey products, i.e., the case of $1\times 1$-matrices; this is true for example in the cases $G=\mathbb{Z}/2\mathbb{Z}\times\mathbb{Z}/2\mathbb{Z}$ (\cite{bks}, Example 7.7) and $G=\mathbb{Z}/3\mathbb{Z}$ (characteristic $3$, \cite{bks}, Example 7.6). In our present case, it is not that easy:
\begin{lemma} Let $k=\mathbb{F}_2$ be the field with $2$ elements. 
For all $a,b,c\in\hat{H}^*(Q_8)$ satisfying $ab=0$ and $bc=0$ we have $0\in\left<a,b,c\right>$.
\end{lemma}
\begin{proof}
By \cite{bks}, Lemma~5.14, the class  $m(a,b,c)$ is contained in the Massey product $\left<a,b,c\right>$. Therefore, it is enough to show that $m(a,b,c)$ is an element of the indeterminacy 
\[ a\cdot\hat{H}^{|b|+|c|-1}(G)+\hat{H}^{|a|+|b|-1}(G)\cdot c \]
for all $a,b,c$. By construction of $m$ it is enough to do so for those triples $(a,b,c)$ and $(sa,b,c)$ with $a,b,c\in\left\{1,x,y,x+y,x^2,y^2,x^2+y^2,x^2y\right\}$ which satisfy $ab=0$ and $bc=0$.

If $|a|,|b|\leq 1$, then $ab=0$ implies $a=0$ or $b=0$ (here we use that $k=\mathbb{F}_2$). If $|b|\geq 2$, then $m(a,b,c)=0$ unless $b\in\{y^2,y^2+x^2\}$ and $a,c\in\{x,x+y\}$, in which case $m(a,b,c)=x^2y$ is divisible by $a$. So we can assume that $|b|=1$ and therefore $|a|\geq 2$ and $|c|\geq 2$, which implies $m(a,b,c)=0$ by Theorem~\ref{mquat8}.

For $m(sa,b,c)$ we have by \eqref{mgleichung}
\[
 m(sa,b,c)=a\, m(s,b,c)+m(a,b,c)s.
\]
We have already seen that the second summand lies in the indeterminacy; the first summand is contained in
\begin{align*}
 a\cdot \hat{H}^{|s|+|b|+|c|-1}(G)=sa\cdot \hat{H}^{|b|+|c|-1}(G)
\end{align*}
and therefore in the indeterminacy.
\end{proof}

\bemerkung Note that the Proposition is not true for arbitrary fields of characteristic $2$: If the field $k$ contains an element $\alpha\in k$ satisfying $\alpha^2+\alpha+1=0$, then the Massey product
\[ \left< \alpha x+y,\alpha^2 x+y,\alpha x+y \right> \]
is defined and does not contain $0$.

\subsection{Generalized quaternions}
The picture changes as soon as we consider generalized quaternion groups $G=Q_{4t}$ with $t\geq 4$. It turns out that there is no module detecting the non-triviality of the canonical element $\gamma_G$.


Let $m$ be as in Theorem~\ref{mquat16}, and write $m=m'+m''$, where $m'$ is defined in \eqref{definitionofmprime}. Notice that $m'$ is a Hochschild cocycle, because it is defined to be $s$-periodic, so it is enough to check the cocycle condition on elements in $\B$. But on these elements, $m'$ agrees with $m$. Hence $m'$ is a cocycle, and so is $m''$. Let $\gamma'$ and $\gamma''$ be the corresponding elements in $\hochschild^{3,-1} \hat{H}^*(G)$. In the next two propositions we will show that, for every module $M$, $\gamma'\cup\id_M=0$ and $\gamma''\cup\id_M=0$ in $\Ext^{3,-1}(M,M)$, respectively. It will then follow that $M$ is a direct summand of a realizable module.

\begin{lemma}\label{gammaprimecupmzero}
For every $\Lambda$-module $M$ we have $\gamma'\cup\id_M=0$.
\end{lemma}
\begin{proof}
Notice that every matrix $A\in\Lambda^{I,J}$ can be uniquely written as a sum
\[ A = A_1 + A_x x + A_y y + A_{x^2} x^2 + A_{y^2} y^2 + A_{x^2y} x^2y, \]
where the six matrices on the right hand side lie in $k[s^{\pm 1}]^{I,J[?]}$. The first step in our proof will be to find a suitable free resolution 
\[ M \leftarrow \Lambda^I \xleftarrow{A} \Lambda^J \xleftarrow{B} \Lambda^K \xleftarrow{C} \Lambda^L \]
of $M$. We begin with the definition of $A$. Let $I$ be a minimal set of generators of the right $\Lambda$-module $M$, that is, $I$ generates $M$ but any proper subset of $I$ does not generate $M$ (in case that $M$ is not finitely generated one has to use Zorn's lemma to prove the existence of $I$). 
The inclusion $I\subseteq M$ induces a surjection $\Lambda^I\rightarrow M$. Let $J$ be a minimal set of generators for the kernel of that map; then we obtain an exact sequence $\Lambda^J\xrightarrow{A}\Lambda^I\rightarrow M$. 
Taking $K$ to be a minimal set of generators for the kernel of $A$, we get a map $\Lambda^K\xrightarrow{B} \Lambda^J$ onto that kernel, and finally we let $L$ be a minimal set of generators for the kernel of $B$ to obtain an exact sequence
\begin{align}\label{freeresolution}
 M \leftarrow \Lambda^I \xleftarrow{A} \Lambda^J \xleftarrow{B} \Lambda^K \xleftarrow{C} \Lambda^L. 
\end{align}
We claim that $A_1=0$. Assume the contrary and let $i\in I$, $j\in J$ be such that $(A_1)_{i,j}\neq 0$.  Then $I-\{i\}$ generates $M$ which contradicts the choice of $I$. Similarly one shows that $B_1=0$ and $C_1=0$, and therefore $B_yC_y = (BC)_{y^2} = 0$.

Now define $W=A_xB_y x + A_{x^2}B_y x^2$ and $V = B_yC_{y^2}y^2$. Then
\begin{align*}
 AV &= A_xB_yC_{y^2} x^3, \\
 WC &= A_xB_yC_x x^2 + A_x\underbrace{B_yC_y}_0 x^2 + A_xB_yC_{x^2}x^3 + A_xB_yC_{y^2}x^3 \\
   & \quad + A_{x^2}B_yC_x x^3 + A_{x^2}\underbrace{B_yC_y}_0 x^3.
\end{align*}
Therefore, $m'(A,B,C) = AV+WC$, and by Proposition~\ref{propexactindeterminacy} we get $\gamma'\cup\id_M=0$.
\end{proof}

\begin{lemma}\label{gammadblprimecupmzero}
For every $\Lambda$-module $M$ we have $\gamma''\cup\id_M=0$.
\end{lemma}
\begin{proof}
We start with a slight modification of the representative $m''$. Put $\B=\{1,x,z,x^2,z^2,x^3\}$, and let us define the function $g$ as follows: For all integers $i$, put
\begin{align*}
 g(s^{-1}x^2,s^ix) &= s^{i-1}z^2, \\
 g(s^{-1}x^2,s^iz) &= s^{i-1}x^2,
\end{align*}
and $g(a,b)=0$ on all other elements $a,b$ in $\{s^ic \mid c\in\B\}$. Then $\tilde{m} = m'' + \partial g$ defines a new representative for the element $\gamma''$. For all $a,b,c\in \B$ and $i,j\geq 1$ we have
\[ \tilde{m}(a,s^ib,s^jc) = m''(a,s^ib,s^jc) + ag(s^ib,s^jc) + g(s^iab,s^jc) + g(a,s^{i+j}bc) + g(a,s^ib)s^jc, \]
and by definition of $m''$ and $g$ each summand on the right hand side vanishes. We also have that
\begin{align*}
 \tilde{m}(s^{-1}a,s^ib,s^jc) &= m''(s^{-1}a,s^ib,s^jc) + s^{-1}a\underbrace{g(s^ib,s^jc)}_0 + \underbrace{g(s^{i-1}ab,s^jc)}_0 \\
 & \quad + g(s^{-1}a,s^{i+j}bc) + g(s^{-1}a,s^ib)s^jc.
\end{align*}
We claim that this is zero if $|a|\geq 2$, $|b|\geq 1$ and $|c|\geq 1$. In that case, we have $|bc|\geq 2$ and therefore $g(s^{-1}a,s^{i+j}bc)=0$, so that it remains to show $m''(s^{-1}a,s^ib,s^jc) = g(s^{-1}a,s^ib)s^jc$, or equivalently
\[ m''(s^{-1}a,b,c)= g(s^{-1}a,b)c. \]
To see this, we consider the several cases for $a$ separately. If $a=x^3$, then 
\[ m''(s^{-1}a,b,c) = s^{-1}x^3 \class(h(b,c)), \]
where $h$ is as in Theorem~\ref{mquat16}. But $|h(b,c)|\geq 1$, so the last expression vanishes, as does $g(s^{-1}a,b)c$. For $a=z^2$ we get
\[ m''(s^{-1}a,b,c) = s^{-1}z^2 \class(h(b,c)), \]
but $|h(b,c)|\geq 2$ or $\class(h(b,c))$ is divisible by $x$, and therefore again the right hand side vanishes. The last case is $a=x^2$ where we need to show
\[ s^{-1}x^2\class(h(b,c)) = g(s^{-1}x^2,b)c. \]
Both sides vanish for degree reasons unless $|b|=|c|=1$, and in that case both sides will equal $s^{-1}x^3$ if $b\neq c$, and $0$ otherwise.

The rest is easy. We start with a free resolution of $M$ as in the proof of Proposition~\ref{gammaprimecupmzero}. We can (and do) assume that the degree $|i|$ of every element $i\in I$ lies in $\{0,1,2,3\}$. Also, we assume that the degree of every element of $J$ lies in $\{-1,0,1,2\}$, the degree of every element of $K$ belongs to $\{-8,-7,-6,-5\}$, and the degree of every element of $L$ is in $\{-15,-14,-13,-12\}$. Then we know that every non-zero entry of $B$ and $C$ is a linear combination of terms of the form $s^ib$ with $i\geq 1$ and $b\in\B$, $|b|\geq 1$. Furthermore, every non-zero entry of $A$ is a linear combination of elements in $\B\cup\{s^{-1}x^2,s^{-1}z^2,s^{-1}x^3\}$. By what we have shown above, $\tilde{m}(A,B,C)=0$, and we are done.
\end{proof}

\bigskip

\begin{small}{\scshape Mathematisches Institut, WWU M\"unster, Einsteinstr. 62, 48149 M\"unster, Germany}

{\itshape Email-address: }\ttfamily{martinlanger@uni-muenster.de}\end{small}


\begin{thebibliography}{1}

\bibitem{AdemMilgram}
Alejandro Adem and R.~James Milgram, \emph{Cohomology of finite groups}, second
  ed., Grundlehren der Mathematischen Wissenschaften [Fundamental Principles of
  Mathematical Sciences], vol. 309, Springer-Verlag, Berlin, 2004.
  MR2035696 (2004k:20109)

\bibitem{bks}
David Benson, Henning Krause, and Stefan Schwede, \emph{Realizability of
  modules over {T}ate cohomology}, Trans. Amer. Math. Soc. \textbf{356} (2004),
  no.~9, 3621--3668 (electronic). MR2055748 (2005b:20102)

\bibitem{Carlson}
Jon~F. Carlson, \emph{Modules and group algebras}, Lectures in Mathematics ETH
  Z\"urich, Birkh\"auser Verlag, Basel, 1996, Notes by Ruedi Suter.
  MR1393196 (97c:20013)

\bibitem{CartEile}
Henri Cartan and Samuel Eilenberg, \emph{Homological algebra}, Princeton
  Landmarks in Mathematics, Princeton University Press, Princeton, NJ, 1999,
  With an appendix by David A. Buchsbaum, Reprint of the 1956 original.
  MR1731415 (2000h:18022)

\bibitem{May}
J.~Peter May, \emph{Matric {M}assey products}, J. Algebra \textbf{12} (1969),
  533--568. MR0238929 (39 \#289)

\end{thebibliography}
\end{document}